\documentclass[12pt]{amsart}

\usepackage{amsfonts,amssymb,amsmath,amsthm}
\usepackage{url}
\usepackage{xcolor}
\usepackage{bbm}

\usepackage{enumerate}
\usepackage{paralist}

\usepackage{stmaryrd}

\theoremstyle{plain}
\newtheorem{theorem}{Theorem}[section]
\newtheorem{lemma}[theorem]{Lemma}
\newtheorem{corollary}[theorem]{\bf Corollary}
\newtheorem{remark}[theorem]{\bf Remark}

\newtheorem{proposition}[theorem]{\bf Proposition}

\newtheorem{claim}{\bf Claim}

\newcommand \Aut{ \mbox{Aut}}
\newcommand \Id{ \mbox{Id}}

\newcommand \dom{{\rm{ dom}}}
\newcommand \rng{{\rm{ rng}}}
\newcommand \age{{\rm{ Age}}}
\newcommand \aut{{\rm{ Aut}}}
\newcommand \acl{{\rm{ acl}}}

\def\acts{\curvearrowright}
\def\f{\mathcal{F}}

\newcommand {\gen}[1]{\left\langle #1 \right\rangle}

\newcommand \NN{\mathbb{N}}

\newcommand \QQ{\mathbb{Q}}

\newcommand \QU{\mathbb{QU}}

\newcommand \ZZ{\mathbb{Z}}
\newcommand \SSS{\mathbb{S}}

\newcommand{\fra}{Fra\"\i ss\'e }

\begin{document}

\title[Automorphism groups and groups of measurable functions]{Automorphism groups of countable structures and groups of measurable functions}
\author[A. Kwiatkowska]{Aleksandra Kwiatkowska}
\address{Institut f\"{u}r Mathematische Logik und Grundlagenforschung, Universit\"{a}t  M\"{u}nster, Einsteinstrasse 62, 48149  M\"{u}nster,
Germany {\bf{and}} Instytut Matematyczny, Uniwersytet Wroc{\l}awski,  pl. Grunwaldzki 2/4, 50-384 Wroc{\l}aw, Poland}
\email{kwiatkoa@uni-muenster.de}
\author[M. Malicki]{Maciej Malicki}
\address{Department of Mathematics and Mathematical Economics, Warsaw School of Economics, al. Niepodleg{\l}o\'sci 162, 02-554, Warsaw, Poland}
\email{mamalicki@gmail.com}

\thanks{The first named author was supported by Narodowe Centrum Nauki grant 2016/23/D/ST1/01097.}


\begin{abstract}

Let $G$ be a topological group  and let $\mu$ be the Lebesgue measure on the interval $[0,1]$. We let $L_0(G)$ to be the topological group of all $\mu$-equivalence classes of $\mu$-measurable functions defined on [0,1] with values in $G$, taken with the pointwise multiplication and the topology  of convergence in measure.
%We will prove several properties of $L_0(G)$ related to the largeness of conjugacy classes
We show that for a Polish group $G$, if $L_0(G)$ has ample generics, then $G$ has ample generics, thus the converse to a result of Ka\"{i}chouh and Le Ma\^{i}tre.

We further study topological similarity classes and conjugacy classes for many groups $\Aut(M)$ and $L_0(\Aut(M))$, where $M$ is a countable structure. We make a connection between the structure of groups generated by tuples, the Hrushovski property, and the structure of their topological similarity classes. 
In particular, we prove the  trichotomy that for every tuple $ \bar{f}$ of $\Aut(M)$, where $M$ is a countable structure such that algebraic closures of finite sets are finite,
either the countable group  $\langle \bar{f} \rangle$ is precompact, or
it is discrete, or the similarity class of $\bar{f}$ is meager, in particular the conjugacy class of $\bar{f}$ is meager.
We prove an analogous trichotomy for groups $L_0(\Aut(M))$.
\end{abstract}

\maketitle

\tableofcontents

\section{Introduction}

Let $G$ be a topological group, and let $\mu$ be the Lebesgue measure on the interval $[0,1]$.
We define $L_0(G)$ to be the topological group of all $\mu$-equivalence
classes of $\mu$-measurable functions on [0,1] with values in $G$. The multiplication in $L_0(G)$ is pointwise,
and the topology is of convergence in measure; a  basic neighbourhood of the identity  is of the form
\[\{ f\in L_0(G)\colon \mu\{x\in [0,1]\colon f(x)\notin V\}<\epsilon\},\]
where $V$ is a neighbourhood of the identity in $G$ and $\epsilon>0$.

%The group $G$ can be identified with the closed subgroup of $L_0(G)$ consisting of all constant maps.
Groups of the form $L_0(G)$ were first considered by Hartman and Mycielski~\cite{HM}, who used them to show that every topological group
can be embedded into a connected one. In fact, $L_0(G)$  is a path-connected and locally-path connected group. In recent years, $L_0(G)$ groups and their generalizations, where instead of a non-atomic probability measure we consider a diffused submeasure,
have attracted attention of a number of researchers.
For example, Glasner~\cite{Gla}, Pestov~\cite{Pe}, Pestov-Schneider \cite{PS}, Farah-Solecki \cite{FS}, and Sabok \cite{Sa}, 
studied  extreme amenability for such groups.
 Solecki~\cite{So} classified all unitary representations of $L_0(\SSS^1)$, where $\SSS^1$ is the circle equipped 
with the usual operation of adding angles. Ka\"{i}chouh and Le Ma\^{i}tre  \cite{KM} gave the first examples of connected Polish groups that have ample generics, all of the form $L_0(G)$. As a matter of fact, this result was the original motivation for our paper.

A {\em Polish group}, that is a separable and completely metrizable topological group, $G$ has {\em ample generics} if it has a {\em comeager diagonal conjugacy class} for every $n$,  that is, if for every $n\geq 1$ the {\em $n$-diagonal conjugation action} of $G$ on $G^n$, given by $g.(h_1,\ldots, h_n)=(gh_1g^{-1},\ldots ,gh_ng^{-1})$, has a comeager orbit. This notion was first studied by Hodges, Hodkinson, Lascar, and Shelah in \cite{HoHo}, and later by Kechris and Rosendal in \cite{KR}. It is a very strong property: a group $G$ with ample generics has the small index property, every homomorphism from $G$ into a separable group is continuous, every isometric action of $G$ on a separable metric space is continuous, and there is only one Polish group topology on $G$.
Surprisingly, many  automorphism  groups do have ample generics, among the examples are the automorphism groups $\Aut(M)$ of the following countable structures $M$:
the random graph (Hrushovski \cite{Hr}), the free group on countably many generators (Bryant and Evans \cite{BrEv}), arithmetically saturated models of true arithmetic (Schmerl \cite{Sch}) or the countable atomless Boolean algebra (Kwiatkowska \cite{Kw}).

 On the other hand, until recently, it was not known whether there exist Polish groups with ample generics that are not isomorphic to the automorphism group of some countable structure (i.e.,  non-archimedean groups). 
Ka\"{i}chouh and Le Ma\^{i}tre \cite{KM} (see also Malicki \cite{Mal}) showed that if a Polish group $G$ has ample generics, then the Polish group $L_0(G)$ (which, being connected, is never non-archimedean) 
has ample generics as well. It is a natural question whether the converse to their result holds. In Section 3, we will give a positive answer to it (Corollary \ref{conv}).  

The notion of  topological similarity  was introduced by Rosendal \cite{R} who used it to give simple proofs of the non-existence of comeager conjugacy classes in groups such as the isometry group of the Urysohn metric space or the group of all Lebesgue measure preserving automorphisms of the interval [0,1]. A tuple $\bar{f}=(f_1,\ldots, f_n)$ in a topological group $G$ is said to be {\em topologically similar} to a tuple $(g_1,\ldots, g_n)$ if the map sending $f_i \mapsto g_i$ extends (necessarily uniquely) to a bi-continuous 
isomorphism between the countable topological groups generated by these tuples. The {\em topological $n$-similarity class} of $\bar{f}$ is then the set of  all  $n$-tuples in $G$ that are topologically similar to it. Clearly, if two tuples are diagonally conjugate, then they are topologically similar.

%Originally, they this concept was used to give a simple proof that groups such as $\Iso(UU)$ or $\Aut(X,\lambda)$ for proving the non-existence of ample generics Clearly, if a Polish group has all its topological $n$-similarity classes meager then it has all its $n$-diagonal conjugacy classes meager. Origin 

In Section 4, we study under what circumstances, a non-archimedean Polish group $G$, %has all  topological similarity classes meager, then
as well as $L_0(G)$, have all  topological similarity classes meager.
We will prove (Theorem \ref{th:MeagSimClass}) that the automorphism group $G=\Aut(M)$ of a countable structure $M$ such that algebraic closures of finite subsets of $M$ are finite, has only meager topological $n$-similarity classes, provided that comeagerly many $n$-tuples in $G$ generate a non-discrete and non-precompact group. As a matter of fact, we obtain a general trichotomy (Theorem \ref{th:Trichotomy}) that relates the structure of subgroups generated by tuples with the structure of their topological similarity classes. Namely, for every tuple $\bar{f}$ in $G$, one of the following holds: $\bar{f}$ generates a precompact group or $\bar{f}$ generates a discrete group or $\bar{f}$ has meager topological similarity class. We also prove a trichotomy (Theorem \ref{th:Hrushovski}) that involves the Hrushovski property and implies that if $G={\rm Aut}(M)$ has ample
generics then either $M$ has the Hrushovski property or the 1-Hrushovski property fails.
% either $M$ has the Hrushovski property, or it does not have the $1$-Hrushovski property or, for some $n$, no $n$-tuple in $G$ has a comeager topological similarity class (and so, in particular, $G$ does not have ample generics.)

If, moreover, there is a dense $n$-diagonal conjugacy class in $G$, the same conditions 
on $G$ as in the statement of Theorem \ref{th:MeagSimClass} imply that all topological $n$-similarity classes in $L_0(G)$ are meager (Theorem \ref{th:MeagSimClassL0}), and we prove that
an analogous trichotomy to the one stated in Theorem \ref{th:Trichotomy}  holds for tuples in $L_0(G)$
(Theorem \ref{th:TrichotomyL0}). We will combine Theorem \ref{th:MeagSimClass} with Lemmas 3.13 and 5.7 in \cite{Sl} about comeagerness of the set of pairs generating a non-discrete group and recover results of Slutsky \cite{Sl} that topological similarity classes of automorphism groups of rationals and of the ordered rational Urysohn space are meager, and we will generalize them to $L_0(G)$, where $G$ is one of these two groups.

\section{Preliminaries}

\subsection{Groups of measurable functions}
Let $X$  be the interval [0,1] and let $\mu$ be the Lebesgue measure. Let $Y$ be a Polish space, we then define 
$L_0(X,\mu; Y)$ to be the set of all ($\mu$-equivalence classes of) measurable (equivalently: Borel) functions from $X$ to $Y$. We equip this space with the convergence of measure topology.

The neighborhood basis at $h\in L_0(X,\mu; Y)$ is
 \[[h,\delta,\epsilon]=\{ g\in L_0(X,\mu; Y) \colon \mu(\{x\in X \colon d(g(x),h(x))<\delta\})>1-\epsilon \}, \]
     where $d$ is a fixed compatible complete metric on $Y$.
This topology does not depend on the choice of the metric $d$ on $Y$.
Then we immediately see that the sets
 \[[h,\epsilon]=\{ g\in L_0(X,\mu; Y) \colon \mu(\{x\in X \colon d(g(x),h(x))<\epsilon\})>1-\epsilon \} \]
also form a neighborhood basis at $h$.
The collections of sets $\{[h,\frac{1}{k},\frac{1}{n}]\colon h\in  L_0(X,\mu; Y), k,n\in\NN\}$ and 
$\{[h,\frac{1}{k}]\colon h\in  L_0(X,\mu; Y), k\in\NN\}$ are both bases of $ L_0(X,\mu; Y)$.
Instead of taking all $h\in L_0(X,\mu; Y)$, we could consider only those in some dense set, e.g., the set of step functions $f_{\bar{a}}$, where $\bar{a}=(a^1, \ldots, a^{n})$ is an $n$-tuple in $Y$, and $f_{\bar{a}}$ is defined by
\[ f_{\bar{a}}(x)=a^k \mbox{ iff } x \in \left[ \frac{k-1}{n}, \frac{k}{n} \right]. \]

%In particular,
%$\{[s,\frac{1}{k},\frac{1}{n}]\colon s \text{ is a step function}, k,n\in\NN\}$  and $\{[s,\frac{1}{k}]\colon s \text{ is a step function}, k\in\NN\}$ 

 The space  $ L_0(X,\mu; Y)$ is metrizable by the metric
\[ \rho(f,g)=\inf\{\epsilon>0\colon \mu(\{ x\in X\colon d(g(x),f(x))>\epsilon\})<\epsilon \}.\]
The metric $\rho$ is complete, and $ L_0(X,\mu; Y)$ is a Polish space.

%\item As proved in \cite{KM}, for every open set $U$ in $G$ and any $\epsilon>0$, the set 
%\[ V_{U,\epsilon}=\{ f\in L_0(X,\mu; Y) \colon \mu(\{x\in X \colon f(x)\in U \})>1-\epsilon  \}  \]  
%is open in $ L_0(X,\mu; Y)$. 

When $Y=G$ is a Polish group, then $ L_0(X,\mu; G)$ is a Polish group as well, with the identity equal to the constant function with value equal to
the identity in $G$ and the multiplication and inverse are given pointwise: $(fg)(x)=f(x)g(x)$ and
$(f^{-1})(x)=f^{-1}(x)$.
A Polish $G$-space $Z$ induces a Polish  $ L_0(X,\mu; G)$-space $ L_0(X,\mu; Z)$, where we act coordinatewise, i.e. $(g.f)(x)=g. (f(x))$.

For a given $B\subseteq Z$, we let
\[  L_0(X,\mu; B)=\{ f\in L_0(X,\mu; Z) \colon f(x)\in B\}.\]
For short, we will often write $L_0(Z)$ instead of $ L_0(X,\mu; Z)$. 

\subsection{Countable structures and non-archimedean groups.}

A \emph{structure} is a set with relations and functions. A structure $M$ is called \emph{ultrahomogeneous} if every isomorphism between finite substructures of $M$ can be extended to an automorphism of $M$. In this paper, we will work only with countable and {\em locally finite} structures, i.e., countable structures all of whose finitely generated substructures are finite.

A Polish group is \emph{non-archimedean} if it admits a neighbourhood basis at the identity consisting of open subgroups. It is well known that non-archimedean Polish groups are exactly those that can be realized as automorphism groups $\Aut(M)$ (equipped with the pointwise convergence topology) of ultrahomogeneous structures $M$ or, equivalently, \fra limits of \fra classes of finite structures. For more details about the \fra construction, and terminology related to it, see \cite[Chapter 7]{Ho}. In particular, we will be using the notion of {\em age}, denoted $\age(M)$, of a structure $M$ (i.e. the class of all finite substructures embeddable in $M$), and its properties such as the joint embedding property or the amalgamation property.

\section{Comeager conjugacy classes in $L_0(G)$}

The main result of this section is the following theorem.
\begin{theorem}\label{comea}
Let $G$ be a Polish group and $Y$ a Polish $G$-space. If there is no comeager orbit in $Y$, then every orbit in the induced Polish $L_0(G)$-space $ L_0(Y)$ is meager.
\end{theorem} 

\begin{corollary}\label{conv}
If $L_0(G)$ has a comeager conjugacy class, then $G$ has a comeager conjugacy class.
\end{corollary}

Since we can naturally identify the groups $L_0(G)^n$ and $L_0(G^n)$, in fact, we obtain a stronger result.

\begin{corollary}
If, for some $n \geq 1$, $L_0(G)$ has a comeager $n$-dimensional diagonal conjugacy class, then $G$ has a comeager $n$-dimensional diagonal conjugacy class. In particular, if $ L_0(G)$ has  ample generics, then $G$ has ample generics.
\end{corollary}

Let a Polish group continuously act on a Polish space.
As each non-meager orbit is comeager in some non-empty open set, and two different orbits cannot intersect, the number of non-meager orbits is bounded by the cardinality of a basis for the topology.  % and, since they are not comeager, none of them is dense in $Y$. %This means that there exists $A_0 \subseteq X$ of positive measure, and a single orbit $B_0 \subseteq Y$ such that $B_0$ is the orbit of $f(x)$ for $x \in A_0$. 
Thus, there can be only countably many non-meager orbits, and every dense and non-meager orbit is comeager. Therefore 
the proof of Theorem \ref{comea} will follow immediately from the following two claims, which we prove in
Sections \ref{c1} and \ref{c2}, respectively.

\begin{claim}\label{meager}
Let $f \in L_0(G)$ be given. Suppose that for all $x \in X$, the orbit of $f(x)$ is meager in the action $G\acts Y$. Then the orbit of $f$ in the induced action $L_0(G)\acts L_0(Y)$ is meager as well. 
\end{claim}

\begin{claim}\label{nonmeager} Let $f \in L_0(G)$ be given.
Suppose that there exist 
$A_0 \subseteq X$ of positive measure, and a non-meager not dense $B_0 \subseteq Y$ such that $B_0$ is the orbit of $f(x)$ for every $x \in A_0$. Then the orbit of $f$ in the induced action $L_0(G)\acts L_0(Y)$ is meager.
%Suppose that there is no comeager orbit in $G\acts Y$. Then all orbits in the induced action $L_0(G)\acts L_0(Y)$ are meager. 
\end{claim}
%Everywhere in this section $X$ and $Y$ will denote Polish spaces, $d$ will be a compatible complete metric on $Y$, and $G$ will be a Polish group.

%The proof of the following lemma is presented in the beginning of the proof of Theorem 6 in \cite{KM}.

We will use the following lemma.
\begin{lemma}
\label{le:Orbit}
Let $Y$ be a Polish $G$-space and let  $ L_0(Y)$ be the induced Polish $ L_0(G)$-space.
 If  $f\in  L_0(Y)$, then 
 \[ L_0(G). f=\{g\in L_0(Y)\colon g(x)\in G. f(x) \text{ for almost all } x\in X\}.\]
\end{lemma} 
\begin{proof}
 The proof goes along the lines of the beginning of the proof of Theorem 6 in \cite{KM}. For completeness, we will sketch it here.
 The inclusion $\subseteq$ is clear. For the opposite inclusion, suppose that a measurable function $g$ is such that 
 $g(x)\in G. f(x) \text{ for almost all } x\in X$. Without loss of generality, $g$ is Borel and for every $x\in X$,  $g(x)\in G. f(x)$.
 Consider the analytic set $S=\{(x,h)\in X\times G\colon g(x)=h. f(x)\}$ and using the Jankov-von Neumann uniformization theorem,
 obtain a measurable map $\phi\in L_0(G)$ whose graph is contained in $S$, that is, $g=\phi. f$, so it belongs to $L_0(G). f$.
 \end{proof}

\subsection{Proof of Claim \ref{meager}}\label{c1}

We start with the following observations.
 
\begin{lemma}
\label{le:BorelCov}
Let $X$ be a Polish space and let $Y$ be a Polish $G$-space.
%\begin{enumerate}
%\item The set NWD of all closed nowhere dense sets in $Y$ is  Borel in $\f(Y)$.
%\item Let $U$ be an open neighbourhood of $e$ in $G$. Then the function that takes  $y\in Y$ and assigns to it $\overline{U. y}$ is Borel.
%\item
For any Borel function $f\colon X\to Y$ such that $G.f(x)$ is meager for each $x\in X$ there exists a Borel function $F\colon X\to \f(Y)^\NN$ such that, for each $x\in X$, $F(x)$ is an increasing sequence of nowhere dense sets whose union covers $G. f(x)$.
%\end{enumerate}
 \end{lemma}

\begin{proof}
First, we show that the set of all closed nowhere dense sets in $Y$ is Borel in the standard Borel space $\f(Y)$ of closed subsetes of $Y$. Indeed, we know that the relation $R(F_1,F_2)$ iff $F_1\subseteq F_2$ is a Borel  subset of $\f(Y)\times \f(Y)$ (see \cite{K}, page 76). Enumerate a countable  basis of open sets  in $Y$ into $\{ U_n \}$. Then
\begin{equation*}
\begin{split}
F \in \mbox{NWD }  \iff  &  \forall_n\exists_m\  U_m\subseteq U_n \text{ and } F\cap U_m=\emptyset \\
  \iff&\forall_n\exists_m \ Y\setminus U_n\subseteq Y\setminus U_m \text{ and } F\subseteq  Y\setminus U_m\\
 \iff&\forall_n\exists_m \ R(Y\setminus U_n, Y\setminus U_m) \text{ and } R(F, Y\setminus U_m).
\end{split}
\end{equation*}

Next, we show that for every open neighborhood $U$ of $e$ in $G$, the function that takes  $y\in Y$ and assigns to it $\overline{U. y}$ is Borel. Indeed, let $(g_n)$ be a dense set in $U.$ As
\[  \overline{U. y}\cap W \neq \emptyset\ \iff\ U. y\cap W \neq \emptyset \ \iff   \  \exists_n g_n.y \in W,\]
we get that \[\{y\in Y\colon\{ \overline{U. y}\cap W \neq \emptyset\}\}\] is Borel for every $W$ be open in $Y$. 
 
Finally, let $(U_n)$ enumerate a countable basis of open sets in $Y$. For each $n$, the function $H_n$ that takes  $y\in Y$ and assigns to it $\overline{U_n. y}$ is Borel,
 and so the function $H\colon Y\to \f(Y)^\NN $ given by $(H(y))(n)=H_n(y)$ is Borel.
 As for every open $U$ the set $\{y\in Y\colon \overline{U. y}\in\text{NWD} \}$ is Borel,
the function $M\colon Y\to \f(Y)$ given by $M(y)=\overline{U_n. y}$ where $n$ is the least such that $U_n. y$, or equivalently
 $\overline{U_n. y}$, is nowhere dense, is Borel. Therefore the function $M_f=M\circ  f$ is Borel.
 Let $\{g_n\}$ enumerate a dense set in $G$.
 Then the function $F\colon X\to \f(Y)^\NN$ given by $F(x)(n)=\{g_1,g_2,\ldots, g_n\}. M_f(x)$ is as required (the map $(F_1,F_2) \mapsto F_1 \cup F_2$ is also Borel).

\end{proof}
 
 Lemma \ref{closed} generalizes Lemma 2 from \cite{KM}, which (after rephrasing) is for $F$ which is constant.

 \begin{lemma}\label{closed}
 Let $F\colon X\to \f(Y)$ be Borel. Then for any $0\leq a\leq 1$ the set 
 \[C=\{g\in L_0(Y)\colon \mu(\{x\in X\colon g(x)\in F(x)\})\geq a\}\]
 is closed in $L_0(Y)$.
 \end{lemma} 
\begin{proof}
Let $(g_n)$ be a sequence in $C$ that converges to some $g\in L_0(Y)$. Since every sequence convergent in measure contains a subsequence that converges
almost surely, without loss of generality we can assume that $(g_n)$ converges almost surely to $g$.

Suppose towards a contradiction that $g\notin C$. That means that for $A_g=\{x\in X\colon g(x)\in F(x)\}$ we have $\mu(A_g)<a$. 
Let $B_k=\{x\in X\setminus A_g\colon U_{\frac{1}{k}}(g(x))\cap F(x)=\emptyset\}$, where $U_\epsilon(x)$ is the ball in $Y$ which is centered at $x$ and has radius $\epsilon$. Then $(B_k)$ is an increasing sequence of measurable sets whose union is $X\setminus A_g$, which implies that $\mu(B_k)$ converges to  $1-\mu(A_g)$.
Take $k_0$ such that $\mu(A_g)+\mu((X\setminus A_g)\setminus B_{k_0})<a$, and take $g_l$ such that 
$\mu(\{x\colon d(g_l(x), g(x))\geq\frac{1}{k_0} \})+\mu(A_g)+\mu((X\setminus A_g)\setminus B_{k_0})<a$.
 Then $g_l\notin C$. A contradiction.

\end{proof}

 %\begin{lemma}
 % Let $F\colon X\to \f(Y)$ be Borel.  Then the set 
 % \[B=\{(x,y) \in X \times Y \colon y\in F(x)\}\]
 % is Borel.
 %\end{lemma}
 %\begin{proof}
 %First note that the set $\{(y,F)\in Y\times \f(Y)\colon y\in F\}$ is Borel. Indeed, we have
 %\[y\notin F \ \iff \ \exists_n (y\in  U_n \text{ and } U_n\cap F=\emptyset).\]
 %Keeping in mind that the graph of a Borel function is a Borel set, we note that $B$ is both analytic and coanalytic:
 %\begin{equation*}
 %\begin{split}
 %(x,y)\in B \ \iff \ & \exists_{A\in\f(Y)} (\text{ if } A=F(x) \text{ then } y\in A) \\ 
 %& \forall_{A\in\f(Y)} (\text{ if } A=F(x) \text{ then } y\in A). \\ 
 %\end{split}
 %\end{equation*}
% \end{proof}

\begin{proof}[Proof of Claim \ref{meager}]
Fix $f\in L_0(Y)$, without loss of generality it is Borel. We show that its orbit $\{g\in L_0(Y)\colon g(x)\in G. f(x) \text{ for almost all } x\in X\}$ (see Lemma \ref{le:Orbit}) is meager in $L_0(Y)$. Using Lemma \ref{le:BorelCov}, fix a Borel $F\colon X\to \f(Y)^\NN$ such that for each $x\in X$, $F(x)$ is an increasing sequence of nowhere dense sets whose union covers $G. f(x)$,
 for readability denote $(F(x))(n)$ by $F_n(x)$.
Fix some $0<\epsilon<1$.
For every $n$ the set 
\[C_{n}= \{g\in L_0(Y)\colon  \mu\{x\in X \colon g(x)\in F_{n}(x)\}\geq 1-\epsilon\}  \]
is closed in $L_0(Y)$, as Lemma \ref{closed} indicates.
%To each $g\in L_0(G). f$, the orbit of $f$,  associate 
%\[ n_g=\min\{n\colon \mu\{x\in X \colon g(x)\in F_n(x)\}> 1-\epsilon\}.  \]

Suppose towards a contradiction that $L_0(G). f$  is non-meager. Then since $L_0(G). f\subseteq \bigcup_n C_n$,
there is $n_0$ such that $C_{ n_0}$ is non-meager. As $C_{n_0}$ is closed and non-meager it has a non-empty interior, and hence there 
is a non-empty open set 
\[[s,\delta]= \{ g\in L_0(Y) \colon \mu(\{x\in X \colon d(g(x),s(x))<\delta\})>1-\delta \} \subseteq C_{n_0},\] where
 %\[V_{h,\delta,\epsilon}=\{ g\in L_0(Y) \colon \mu(\{x\in X \colon d(g(x),h(x))<\epsilon\})>1-\delta  \}\subseteq C_{n_0},\]
 $s$ is a step function.
 Consider the Borel set with open vertical sections
 \[ W=\{ (x,y)\colon d(s(x), y)<\delta\}.\]
 
 %The set 
 %\[ \{ B=(x,y)\in X\times Y\colon y\in F_n(x)\}\]
 %is analytic (?) and therefore it has the Baire property, has all vertical sections nowhere dense, and therefore by the Kuratowski-Ulam theorem
 %there is a comeager set $A\subseteq Y$ such that for every $a\in A$ the horizontal section of $B$ at $a$ is meager. 
%Pick any $a\in A\cap U$ and note that $\bar{a}\in V_{U,\delta}$

But then the Borel set $B=\{(x,y)\colon y\in F_{n_0}(x)\}\cap W$, viewed as a subset of $W$, has all vertical sections meager.
By the uniformization theorem for Borel sets with all sections non-meager, or by the Jankov-von Neumann uniformization theorem (see \cite{K}), 
there exists a measurable function $h\colon X\to Y$ whose graph is contained
in $W\setminus B$. In particular, $h\in [s,\delta]$ and
$h\notin C_{n_0}$, which gives a contradiction.

\end{proof}

%\begin{remark}
%It seems like above it does not matter AT ALL what $\epsilon $ is, as long as it is between 0 and 1. 
%\end{remark}

\subsection{Proof of Claim \ref{nonmeager}}\label{c2}

Let us start with some observations.

%\begin{remark}
%Let $Y$ be a Polish $G$-space. Then every orbit which is non-meager, and is not comeager, cannot be dense.
%{\rm Let $Y$ be a Polish $G$-space. If an orbit is dense and non-meager then it is comeager.}
%\end{remark}

\begin{lemma}\label{kreski}
Suppose that a set $B$ is not dense in $Y$. Then $L_0(B)$ is nowhere dense.
\end{lemma}
\begin{proof}
As the closure of $B$ is also not dense in $Y,$ without loss of generality, we can assume that $B$ is closed.
Then the set $L_0(B)$  is closed. Indeed, it follows from the fact that every sequence convergent in measure has a subsequence which converges almost surely.
It is therefore enough to show that $L_0(B)$ does not contain any open set. 

 %Let $B(p,r)\subseteq Y$  be an open ball disjoint from $B$. 
 Take $f\in L_0(B)$ and $p\notin B$, and let
 $\epsilon>0$. We show that $[f,\epsilon]$ is not contained in  $L_0(B)$. 
 For that pick any Borel set $A\subseteq X$ satisfying $0<\mu(A)<\epsilon$ and take 
 \begin{equation*}
 g(x)=
 \begin{cases}
 f(x) & \text{ if } x\notin A\\
 p & \text{ if } x\in A.
 \end{cases}
 \end{equation*}
 Then $g\in [f,\epsilon]$ and $g\notin  L_0(B)$. 
\end{proof}

\begin{proof}[Proof of Claim \ref{nonmeager}]
Fix $f\in L_0(Y)$, without loss of generality it is Borel. We show that its orbit $\{g\in L_0(Y)\colon g(x)\in G. f(x) \text{ for almost all } x\in X\}$ is meager in $L_0(Y)$.

%If for almost all $x$, orbits of $f(x)$ are meager, we are done by Claim 1. 

%We list those orbits as $B_1, B_2, B_3, \ldots $, and let $A_i=\{x\in X\colon f(x)\in B_i\}$, and  $A_0=X\setminus (\bigcup_{i>0} A_i)$. 
%\blu{Dodalem, bo moze sie zdarzyc, ze $\mu(A_i)=0$} 
%Without loss of generality, by considering only $A_i$'s of positive measure,

If $X \setminus A_0$ has measure $0$, the orbit of $f$ is contained in $L_0(A_0,\mu_0; B_0)$, and we are done by Lemma \ref{kreski}. Otherwise, put $A_1=X \setminus A_0$, $B_1=Y \setminus B_0$. We have that for $i=0,1$, $\mu_i=\frac{\mu\restriction A_i}{\mu(A_i)}$ is a probability measure on $A_i$. The space $L_0(Y)$ can be identified in a natural way with the space
$\prod_i L_0(A_i,\mu_i; Y)$;  to  $f\in L_0(Y)$ corresponds $(f\restriction A_i)\in \prod_i L_0(A_i,\mu_i; Y)$. 
%For that we are checking that $f_n\to f$ if and only if for every $i$, $f_n\restriction A_i \to  f\restriction A_i$.
Then the orbit of $f$ is contained in
\[\prod_i L_0(A_i,\mu_i; B_i).\]
As $L_0(A_0,\mu_0; B_0)$ is nowhere dense, by the Kuratowski-Ulam theorem, its product with $L_0(A_1,\mu_1; B_1)$ is a meager set, and so is the orbit of $f$.
\end{proof}

\section{Topological similarity classes}

Let $G$ be a topological group. For an $n$-tuple $\bar{f}=(f_1,\ldots,f_n)$, we let $\gen{\bar{f}}$  denote the countable subgroup of $G$ generated by the set $\{f_1,\ldots,f_n\}$. Recall from the introduction that the topological $n$-similarity class of $\bar{f}$ is the set of all the $n$-tuples $\bar{g}=(g_1,\ldots,g_n)$ in $G$ such that the map sending $f_i \mapsto g_i$ extends (necessarily uniquely) to a bi-continuous isomorphism between the subgroups $\gen{\bar{f}}$ and $\gen{\bar{g}}$ of $G$ generated by these tuples. Observe that topological similarity classes, being coanalytic sets, always have the Baire property.

Let $F_n=F_n(s_1, \ldots, s_n)$ denote the free group on $n$ generators $s_1, \ldots, s_n$. By a reduced word in $F_n$, we mean a sequence of generators $s_{i_0}\ldots s_{i_k}$ with no element followed by its inverse. 
For $u,v \in F_n$, let $uv$ be the word obtained by concatenating $u$ and $v$. For $w \in F_n$, and an $n$-tuple $\bar{f}$ in $G$, the evaluation $w(\bar{f})$ denotes the element of $G$ obtained from $w$ by substituting $f_i$ for $s_i$, and performing the group operations on the resulting sequence.

By a \emph{partial automorphism} of a structure $M$, we mean a mapping $p\colon A \rightarrow B$, where $A, B \subseteq M$ are finite, that can be extended to an automorphism of $M$. Note that we do not require $A,B$ to be substructers, so, in fact, partial automorphisms in this sense are just names for elements of an open basis in $\Aut(M)$.
For an $n$-tuple of partial automorphisms $\bar{p}$ of $M$, we denote by $[\bar{p}]$  the open set of all those $\bar{g} \in \Aut(M)^n$ that extend $\bar{p}$. If $G$ (respectively $L_0(G)$) is equipped with a metric,  $g \in G$  (respectively $L_0(G)$),  and 
$\epsilon>0$, then $U_{\epsilon}(g)$ denotes the $\epsilon$-neighborhood of $g$ in $G$  (respectively $L_0(G)$). 

\begin{lemma}
\label{le:NonDisc}
Let $G$ be a Polish group, and, for a fixed $n $, let $\{u_k\}$ be a sequence of reduced words in $F_n$. For every $n$-tuple $\bar{f}$ in $G$ such that $\gen{\bar{f}}$ is non-discrete in $G$ there exists a sequence of reduced words $w_k=v_ku_k$ such that $w_k(\bar{f}) \rightarrow e$.
\end{lemma}

\begin{proof}
If $n=1$, the conclusion is clear.

For the rest of the proof, suppose that $n\geq 2$. Fix a neighborhood basis at the identity $\{V_k\}$ in $G$. For $\bar{f} \in G^n$ such that $\gen{\bar{f}}$ is non-discrete, fix reduced words $\{v'_k\}$ in $F_n(s_1, \ldots, s_n)$ such that $v'_k(\bar{f}) \rightarrow e$. 
Note that $s_i^{-1}v'_ks_i(\bar{f}) \rightarrow e$ and $s_iv'_ks_i^{-1}(\bar{f}) \rightarrow e$, 
for $i=1,\ldots, n$. Put in one sequence all $\{s_i^{-1} v'_k s_i\}$, $\{s_i v'_k s_i^{-1}\}$,  as well as $\{v'_k\}$. 
 %and call the resulting sequence
%$\{v''_k\}$.

Fix $k_0 \in \NN$, and suppose that the word $u_{k_0}$ is of the form $s_1 u'_{k_0}$. There exists $k_1 \in \NN$ such that
 %the reduction of the word $v''_{k_1}$ ends in  $s_1$ or $s_1^{-1}$ and
%
\[ u_{k_0}^{-1} v'_{k_1} u_{k_0}(\bar{f}) \in V_{k_0}\]
as well as 
\[ u_{k_0}^{-1} s_2^{-\epsilon}v'_{k_1}s_2^{\epsilon} u_{k_0}(\bar{f}) \in V_{k_0}\]
for $\epsilon=-1,1$.

At least one of $u_{k_0}^{-1} s_2^{-1}v'_{k_1}s_2$, $u_{k_0}^{-1} s_2v'_{k_1}s_2^{-1}$, or $u_{k_0}^{-1} v'_{k_1}$, after reducing, does not end in $s_1^{-1}$
(one has to check a number of cases here: the reduction of $v'_{k_1}$ can start in $s_1$, $s_1^{-1}$, $s_2$, $s_2^{-1}$ or in $s_i^{\epsilon}$, $i>2$,
and it can end in $s_1$, $s_1^{-1}$, $s_2$, $s_2^{-1}$ or in $s_i^{\epsilon}$, $i>2$).
Call that reduced word $v_{k_0}$. Then the word $v_{k_0} u_{k_0}$ is reduced.

%Let  $v_{k_0}$ be the reduction of the word $u_{k_0}^{-1}s_2^{-1}v''_{k_1}s_2$. It has to end in $s_2$, and hence the word $v_{k_0}u_{k_0}$ is reduced.

In an analogous manner we deal with the case that $u_{k_0}=s_i^\epsilon u'_{k_0}$ for other $i \leq n$, $\epsilon=-1,1$. Thus, we obtain reduced words $w_k=v_ku_k$, $k \in \NN$, such that $w_k(\bar{f}) \rightarrow e$.
\end{proof}

For a structure $M$, and finite subset $A \subseteq M$, let $\aut_A(M)$ denote the subgroup of all $f \in \aut(M)$ that pointwise stabilize $A$. The \emph{algebraic closure} of $A$, usually denoted by $\acl_M(A)$ is the set of all elements in $M$ whose orbits are finite under the evaluation action of $\aut_A(M)$ (see \cite{Ho} for more information on this notion). We say that $A$ is \emph{algebraically closed} if $\acl_M(A)=A$. It is easy to see that $\acl_M(\acl_M(A))=\acl_M(A)$, and that for every algebraically closed $A \subseteq M$, and $a \in M \setminus A$, the orbit of $a$ under the action of $\aut_A(M)$ is infinite.

We say that $M$ has \emph{no algebraicity} if every finite $A \subseteq M$ is algebraically closed.
Note that $M$ has no algebraicity iff for every finite $A \subseteq M$, and every finite set $B\subseteq M \setminus A$, 
the orbit of $B$ under the action of  $\Aut_A(M)$  is infinite.

It is easy to see that a non-archimedean Polish group $G$ is precompact iff every orbit under the evaluation action of $G$ is finite.

\begin{lemma}
\label{le:NonConv}
Let $M$ be a structure such that algebraic closures of finite subsets of $M$ are finite, and let $G=\Aut(M)$. Suppose that, for some $n$ and open $U \subseteq G^n$, there are comeagerly many $n$-tuples in $U$ generating a non-precompact subgroup of $G$. Then every $n$-tuple whose class of topological similarity is non-meager in $U$ generates a discrete group.
\end{lemma}

\begin{proof}
%Suppose that there are comeagerly many $n$-tuples in $G$ generating a non-precompact subgroup of $G$.
We will show that for every sequence $\{w_k\}$ of reduced words in $F_n$ such that for every reduced word $u$ there exist infinitely many $k$ with $w_k=v_ku$, the set
\[ P= \{ \bar{f} \in U \colon w_k(\bar{f}) \not \rightarrow e \} \]
is comeager in $U$. To see that this implies the lemma, suppose that an $n$-tuple $\bar{f}$ generates a non-discrete group. By Lemma 4.1, there exists a sequence $\{w_k\}$ as above, and such that $w_k(\bar{g}) \rightarrow e$ for every $\bar{g}$ in the similarity class $C$ of $\bar{f}$. But then $C$ must be meager in $U$. 

Fix a sequence $\{w_k\}$ as above. As topological similarity classes have the Baire property, it suffices to show that for every open, non-empty $V \subseteq U$ there exists open, non-empty $V_0 \subseteq V$ such that $w_k(\bar{g}) \not \rightarrow e$ for comeagerly many $\bar{g} \in V_0$.

Fix an open, non-empty $V \subseteq U$. Because $\gen{\bar{f}}$ is non-precompact for comeagerly many $n$-tuples $\bar{f}$ in $V$, there exists $a \in M$ such that the orbit $\gen{\bar{f}}. a$ under the evaluation action of $\gen{\bar{f}}$ on $M$ is infinite for non-meagerly many $n$-tuples $\bar{f}$ in $V$. Thus, there exists an $n$-tuple $\bar{p}$ of partial automorphisms of $M$ such that $[\bar{p}] \subseteq V$, and the orbit of $a$ is infinite for comeagerly many tuples $\bar{f}$ in $[\bar{p}]$. In particular, for every $n$-tuple $\bar{q}$ of partial automorphisms extending $\bar{p}$, there exists $u \in F_n$ such that $u(\bar{q})(a)$ is not defined.

Fix $k \in \NN$, and let 
\[ P_k=\{ \bar{g} \in [\bar{p}]\colon\, \exists l \geq k \ w_l(\bar{g})(a) \neq a \}. \]
We show that $P_k$ is dense in $[\bar{p}]$. Fix an $n$-tuple of partial automorphisms $\bar{q}$ extending $\bar{p}$. 
Our definition of the partial automorphism immediately implies that any  (surjective) partial automorphism $r\colon A\to B$ extends uniquely to
a (surjective) partial automorphism $r\colon {\rm acl}(A)\to {\rm acl}(B)$. 
We can therefore assume that the domain, and thus the range, of every partial automorphism in $\bar{q}$ is algebraically closed. Fix $u \in F_n$ such that $u(\bar{q})(a)$ is not defined. 

Find $l \geq k$ with $w_l=vu$, and write $w_l(\bar{q})=z_1 \ldots z_m$. 
 Let $i_0 \leq m$ be the largest number such that $z_{i_0} \ldots z_m(a)$ is not defined. 

We show by induction on $i_0=1,\ldots,m$ that if $\bar{q}$ and $z_1,\ldots,z_m$ are such that $w_l(\bar{q})=z_1 \ldots z_m$ and
$i_0$ is the largest number such that $z_{i_0} \ldots z_m(a)$ is not defined then there is 
an $n$-tuple $\bar{r}$ of partial automorphisms extending $\bar{q}$ such that $w_l(\bar{r})(a)$ is defined and not equal to $a$.

If $i_0=1$, we use the fact that $\dom(z_{i_0})$ is algebraically closed and $a \not  \in \dom(z_{i_0})$ to extend  $z_{i_0}$ to $z'_{i_0}$, by setting $z'_{i_0}(b)=c$, where $b=z_{i_0+1}\ldots z_m(a)$ and $c \neq a$, and extending the domain to the algebraic closure.
Extend all other $z_i$ of the form $z_{i_0}$ or $z_{i_0}^{-1}$ accordingly. This gives the required $\bar{r}$.

If $i_0>1$, also using algebraic closedness of $\dom(z_{i_0})$, we extend $z_{i_0}$ to $z'_{i_0}$ by setting $z'_{i_0}(b)=c$, where $b=z_{i_0+1}\ldots z_m(a)$, $c \neq b$, and $c$ is such that $z_{i_0-1}(c)$ is not defined, and extending the domain to the algebraic closure. Extend all other $z_i$ of the form
 $z_{i_0}$ or $z_{i_0}^{-1}$ accordingly, and set $z'_i=z_i$ for the remaining $i$'s.
 Note that, since $c\neq b$, even if we had $z_{i_0-1}=z_{i_0}$, then $z'_{i_0-1}(c)$ is still not defined.
 Therefore $i_0-1$ is the largest number such that $z'_{i_0-1}z'_{i_0} \ldots z'_m(a)$ is not defined.  
We apply the inductive assumption to $z'_1,\ldots,z'_m$ and get the required $\bar{r}$, which finishes this step of the induction.

Clearly, each $P_k$ is open, so $P \supseteq \bigcap_k P_k$ is comeager in $[\bar{p}]$. Thus, $V_0=[\bar{p}]$ is as required.

%Note that the same argument shows that if there are non-meagerly many $n$-tuples in $U$ generating a non-precompact subgroup, then $P$ is non-meager.

\end{proof}

%The trichotomy below in particular implies that for a structure $M$ with no algebraicity, and given $n $, if 
%there are comeagerly  many $n$-tuples in $\Aut(M)$ generating a non-precompact and non-discrete subgroup, we have that 
%every $n$-diagonal conjugacy class in $\Aut(M)$ is meager.

The above lemma yields a general condition implying meagerness of topological similarity classes.

\begin{theorem}
\label{th:MeagSimClass}
Let $M$ be a structure such that algebraic closures of finite subsets of $M$ are finite, and let $G=\Aut(M)$. Suppose that, for some $n$, there are comeagerly many $n$-tuples in $G$ generating a non-precompact and non-discrete subgroup. Then each $n$-dimensional class of topological similarity in $G$ is meager.
\end{theorem} 

As a matter of fact, we  obtain the following trichotomy.

\begin{theorem}
\label{th:Trichotomy}
Let $M$ be a structure such that algebraic closures of finite subsets of $M$ are finite, and let $G=\Aut(M)$. Then for every $n$ and $n$-tuple $\bar{f}$ in $G$:
\begin{itemize}
\item $\langle \bar{f} \rangle$ is precompact, or
\item $\langle \bar{f} \rangle$ is discrete, or
\item the similarity class of $\bar{f}$ is meager.
\end{itemize}
\end{theorem}

\begin{proof}
Fix an $n$-tuple $\bar{f}$ in $G$, and suppose that its similarity class $C$ is non-meager. Then it is comeager in some non-empty open $U \subseteq G^n$. By the last statement of Lemma \ref{le:NonConv}, either $\bar{f}$ generates a discrete group or the assumptions of the lemma are not satisfied, i.e., there are non-meagerly many tuples in $U$ that generate a precompact group. But then the set of such tuples must intersect $C$, which implies that $\bar{f}$ also generates a precompact group.
\end{proof}

\begin{remark}
\label{re:HruCompact}
{ \rm We would like to point out that in the above results, even though we allow $n=1$, in fact, only values larger than $1$ provide any new insight. This is because in all non-archimedean Polish groups, subgroups generated by elements must be either discrete or precompact (see Lemma 5 in \cite{Ma1}).}
\end{remark}

Recall that an ultrahomogeneous structure $M$ has the \emph{Hrushovski property} if for every $n \in \NN$, $A \in \age(M)$, and tuple $(p_1, \ldots, p_n)$ of partial automorphisms of $A$, there exists $B \in \age(M)$, such that $A \subseteq B$, and every $p_i$ can be extended to an automorphism of $B$. Similarly, $M$ has the $n$-Hrushovski property if the above holds for the given $n$. Clearly,  $M$ has the Hrushovski property if and only if it has the $n$-Hrushovski property for all $n \in \NN$.

It is well known that the existence of ample generics in the automorphism group of an ultrahomogeneous structure $M$ is implied by the Hrushovski property, provided that sufficiently free amalgamation is present in $\age(M)$ (in case of actual free amalgamation, it is implicit in the results from \cite{HoHo}). This is true for the random graph (because finite graphs can be freely amalgamated)
 %the random tournament (because finite tournaments can be strongly$^+$ amalgamated, see Section 5) 
 or the rational Urysohn space (because finite metric spaces can be metrically-free amalgamated). On the other hand, there exist groups with ample generics that are of a different type. For example, the automorphism group of the countable atomless Boolean algebra (equivalently, the homeomorphism group of the Cantor space) has ample generics and a generic automorphism in this group generates a discrete subgroup, therefore the countable atomless Boolean algebra does not even have the $1$-Hrushovski property. 
%It turns out that Lemma \ref{le:NonConv} sheds more light on this phenomenon.

The following lemma can be extracted from the proof of Proposition 6.4 in \cite{KR}, where it is shown that the Hrushovski property is equivalent to the existence of a countable chain of compact subgroups whose union is $\Aut(M)$.
\begin{lemma}
\label{le:HruPreComp}
Let $M$ be an ultragomogeneous structure. Then, for every $n$, $M$ has the $n$-Hrushovski property if and only if densely many $n$-tuples in $\Aut(M)$ generate a precompact subgroup of $\Aut(M)$ if and only if comeagerly many $n$-tuples in $\Aut(M)$ generate a precompact subgroup of $\Aut(M)$.
\end{lemma}

\begin{proof}
Fix $n \in \NN$.
The set of $n$-tuples in $\Aut(M)$ that generate a precompact subgroup of $\Aut(M)$ is a $G_\delta$,  which implies that the second and the third statements are equivalent.

Suppose that $M$ has the $n$-Hrushovski property, and fix an $n$-tuple $\bar{p}$ of partial automorphisms of $M$. Then $\bar{p}$ can be extended to an $n$-tuple $\bar{f}$ of automorphisms of $M$ such that all orbits under the evaluation action of $\gen{\bar{f}}$ are finite, i.e., $\gen{\bar{f}}$ is precompact. 
 Indeed, using the Hrushovski property and the ultrahomogeneity of $M$ we can find a sequence $A_1\subseteq A_2\subseteq \ldots $
  of finite substructures of $M$, whose union is $M$, and partial automorphisms  $\bar{p}\subseteq \bar{p}_1\subseteq \bar{p}_2\subseteq\ldots $
  such that the domain and the range of every function in $\bar{p}_i$ is equal to $A_i$. Then $\bar{f}=\bigcup_i \bar{p}_i$ is as required.
 This shows that  densely many $n$-tuples in $\Aut(M)$ generate a precompact subgroup of $\Aut(M)$.

On the other hand, if $M$ does not have the $n$-Hrushovski property, there exists an $n$-tuple $\bar{p}$ of partial automorphisms that cannot be extended to such $\bar{f}$, which means that every $n$-tuple $\bar{f}$ of automorphisms of $M$ extending $\bar{p}$ is non-precompact.
\end{proof}

Theorem \ref{47} implies that if the automorphism group $\Aut(M)$ of an ultrahomogeneous structure $M$ such that algebraic closures of finite sets are finite has ample generics then either $M$ has the Hrushovski property or otherwise the 1-Hrushovski property fails.

\begin{theorem}\label{47}
\label{th:Hrushovski}
Let $M$ be an ultrahomogeneous structure such that algebraic closures of finite subsets of $M$ are finite. Then one of following holds:
\begin{enumerate}
\item $M$ has the Hrushovski property,
\item $M$ does not have the $1$-Hrushovski property,
\item there exists $n$ such that none of $n$-dimensional topological similarity classes in $\Aut(M)$ is comeager. In particular, $\Aut(M)$ does not have ample generics.
\end{enumerate}
\end{theorem}

\begin{proof}
Suppose that $M$ does not satisfy (1) and (2), i.e., there exists $n$ such that $M$ does not have the $n$-Hrushovski property but it has the $1$-Hrushovski property. Then, by Lemma \ref{le:HruPreComp}, there are comeagerly many automorphisms in $\Aut(M)$ generating a precompact group. Moreover, it is not hard to see that our assumption that algebraic closures of finite subsets of $M$ are finite, implies that there are densely many (and so comeagerly many) automorphisms generating an infinite group. It follows that there are comeagerly many automorphisms in $\Aut(M)$ generating a non-discrete group, and thus that there are also comeagerly many $n$-tuples generating a non-discrete group.

Now, again by Lemma \ref{le:HruPreComp}, there is a non-empty, open $U \subseteq \Aut(M)^n$ such that comeagerly many $n$-tuples in $U$ generate a non-precompact group. Thus, there are comeagerly many $n$-tuples in $U$ generating a non-discrete and non-precompact group. By Lemma \ref{le:NonConv}, every $n$-dimensional topological similarity class in $\Aut(M)$ is meager in $U$, so, in particular, none of $n$-dimensional topological similarity classes in $\Aut(M)$ is comeager.
\end{proof}

For a property $P$ we say that a {\em generic} $n$-tuple has the property $P$ if $\{\bar{f}\in G^n\colon \bar{f} \text{ has the property } P\}$ is comeager in $G^n$.

\begin{corollary}\label{thm1}
Suppose $G = \Aut(M)$ is the automorphism group of an ultrahomogeneous structure $M$ such that algebraic closures of finite subsets of $M$ are finite. Assume also that $G$ has ample generics or just comeager $n$-similarity classes for every $n$. Then either
\begin{enumerate}
\item $M$ has the Hrushovski property or
\item for every $n$ the generic $n$-tuple $\bar{f}$ generates a discrete subgroup of $G$.
\end{enumerate}
\end{corollary}

\begin{proof}
Suppose that $M$ does not have the Hrushovski property. By Theorem \ref{th:Hrushovski}, $M$ does not have the $1$-Hrushovski property, therefore for every $n$, the generic $n$-tuple must generate a non-precompact group. But then, by Theorem \ref{th:Trichotomy}, the generic $n$-tuple generates a discrete group.
\end{proof}

\begin{remark}
{ \rm It was proved by Rosendal (unpublished) that the generic $2$-tuple in the automorphism group of the countable atomless Boolean algebra generates a discrete group. The above corollary gives a strengthening of this result.}
\end{remark}

%Note that the above considerations provide a general condition implying the equivalence of the existence of ample generics, the existence of comeager similarity classes, and the Hrushovski property. We can formulate it as the following meta-statement.

%We review the free amalgamation property in the next section.
Corollary \ref{thm1} and Theorem 6.2 from \cite{KR} imply the following corollary.

\begin{corollary}\label{thm2}
Suppose that $M$ is an ultrahomogeneous structure such that algebraic closures of finite subsets of $M$ are finite and $\age(M)$ has the free amalgamation property. Assume that, for some $n$, the generic $n$-tuple generates a non-discrete subgroup of $\Aut(M)$. Then the following are equivalent:
\begin{itemize}
\item $M$ has the Hrushovski property,
\item $\Aut(M)$ has ample generics,
\item there exists a comeager $n$-similarity class in $\Aut(M)$ for every~$n$.
\end{itemize}
\end{corollary}
\begin{proof}
Hrushovski property together with  the free amalgamation property easily imply that (ii) of Theorem 6.2 from \cite{KR} is satisfied
 (cf. pages 332-333 in \cite{KR}). Therefore $\Aut(M)$ has ample generics, and hence 
there exists a comeager $n$-similarity class in $\Aut(M)$ for every $n$. 

On the other hand, if there exists a comeager $n$-similarity class in $\Aut(M)$ for every $n$, 
Corollary \ref{thm1} immediately implies that $M$ has the Hrushovski property.
\end{proof}
\begin{remark}
{\rm There are structures $M$ without the free amalgamation property, for which amalgamation is very canonical, and by essentially the same argument   
the Hrushovski property implies that $\Aut(M)$ has ample generics. An example of such an $M$ is  the rational Urysohn space (see pages 332-333
in \cite{KR}). }
\end{remark}  

%For a property $P$ we say that a generic $n$-tuple has the property $P$ if $\{\bar{f}\in G^n\colon \bar{f} \text{ has the property } P\}$ is comeager in $G^n$.

%In the case we have ample generics, we get a dichotomy that may of independent interest.

%\blu{I changed ample generics into comeager $n$-diagonal conjugacy class}
%\begin{corollary}\label{cor:dichotomy}
%Let $M$ be a countable structure such that the algebraic closure of every finite subset of $M$ is finite, and $G=\Aut(M)$ has ample generics.
% Then, for every $n$,  the generic $n$-tuple in $G$ generates a group that is either precompact or discrete.
%\end{corollary}

%\begin{remark} 
%{ \rm  Theorem \ref{th:Hrushovski} implies that in the automorphism group of the countable atomless Boolean algebra, there must be $n$ such that On the other hand, Rosendal (unpublished) showed that the generic 2-tuple
%in the homeomorphism group of the Cantor set generates a discrete group.}
%\end{remark}

Below, $\QQ$ denotes the ordered set of rational numbers, and $\QU_{\prec}$ denotes the randomly ordered rational Urysohn space, i.e. the Fra\"{i}ss\'{e} limit of the class of linearly ordered finite metric spaces with rational distances. Clearly, for $M=\QQ$ or $M=\QU_{\prec}$, $M$ has no algebraicity, therefore every finite subset of $M$ is algebraically closed, also every non-identity element in $\Aut(M)$ generates a discrete group. This implies that every element, and hence also every tuple of elements, in $\Aut(M)$, generates a 
non-precompact group. It was proved in \cite[Lemma 3.13]{Sl} and \cite[Lemma 5.7]{Sl} that the set of pairs generating a non-discrete group is comeager in $\Aut(M)^2$. Thus, we get

\begin{corollary}[Slutsky]
\label{co:Slutsky}
Every $2$-dimensional class of topological similarity in $\Aut(M)$, where $M=\QQ$ or $M=\QU_{\prec}$, is meager.
\end{corollary}  

%It is an open question whether the automorphism group $\Aut(P)$ of the random poset $P$, i.e., the Fra\"{i}ss\'{e} limit of the class of finite posets, has ample generics. In light of Theorem \ref{th:MeagSimClass}, the following would give a negative answer to it:

%\begin{question}
%Is it the case that, for some $n \geq 2$, there are comeagerly many $n$-tuples in $\Aut(P)$ generating a non-discrete group?
%\end{question}

We will now study topological similarity classes in $L_0(G)$. 

\begin{lemma}
\label{le:NonConvL0}
Let $M$ be a structure such that algebraic closures of finite subsets of $M$ are finite, and let $G=\Aut(M)$. Suppose that, for some $n$, there are comeagerly (non-meagerly) many $n$-tuples in $G$ generating a non-precompact subgroup. Then every $n$-tuple in $L_0(G)$ with a non-meager (comeager) class of topological similarity generates a discrete group.
\end{lemma}

\begin{proof}
As in the proof of Lemma \ref{le:NonConv}, it suffices to show that for every sequence $\{w_k\}$ of reduced words in $F_n$ such that for every reduced word $u$ there exist infinitely many $k$ with $w_k=v_ku$, the set
\[ P= \{ \bar{f} \in L_0(G)^n \colon w_k(\bar{f}) \not \rightarrow e \} \]
is comeager (non-meager) in $L_0(G)^n$.

Suppose that there are comeagerly many $n$-tuples in $G$ generating a non-precompact subgroup of $G$.
Fix a sequence $\{w_k\}$ as in the statement of the lemma. Similarly to the proof of Lemma \ref{le:NonConv}, we show that for every open, non-empty $V \subseteq L_0(G)^n$ there exists $V_0 \subseteq V$ such that $w_k(\bar{f}) \not \rightarrow e$ for comeagerly many $\bar{f} \in V_0$.

Fix a non-empty open $V \subseteq L_0(G)^n$. Fix $\bar{f} \in V$, and $0<\epsilon<1$ such that $U_\epsilon(\bar{f}) \subseteq V$. Without loss of generality, we can assume that the tuple $\bar{f}$ consists of Borel mappings. As in Lemma \ref{le:NonConv}, for every $x \in X$ there is $a_x \in M$, and an $n$-tuple $\bar{p}_x$ of partial automorphisms of $M$ such that $[\bar{p}_x] \subseteq U_\epsilon(\bar{f}(x)) \subseteq G$ for every $x \in X$, and the orbit $\gen{\bar{g}}. a_x$ is infinite for comeagerly many $\bar{g} \in [\bar{p}_x]$. Observe that the functions $x \mapsto a_x$, $x \mapsto \bar{p}_x$ can be taken to be Borel. Indeed, let $\{a_m \}$ be an enumeration of $M$, and let $\{\bar{p}_n\}$ be an enumeration of all partial automorphisms of $M$. Then we can put $a_x=a_{m_x}$, $\bar{p}_x=\bar{p}_{n_x}$, where
\[ m_x=\min \{ m \in  \NN\colon \exists^*_{\bar{g} \in G} \ ( \bar{g} \in U_\epsilon(\bar{f}(x)) \mbox{ and } \left| \gen{ \bar{g}}. a_m \right|=\infty) \}, \]
\[ n_x=\min \{ n \in \NN \colon [\bar{p}_n] \subseteq U_\epsilon(\bar{f}(x)) \mbox{ and }
							\forall^*_{\bar{g} \in G} \ (\bar{g} \in [\bar{p}_m] \mbox{ and } \left| \gen{ \bar{g} }. a_{m_x} \right|=\infty) \}. \] 
By the Montgomery-Novikov theorem, category quantifiers over Borel sets yield Borel sets, so the above formulas define Borel functions. In particular, the set 
\[ \{ (x,\bar{g}) \in X \times G^n\colon \bar{g} \in [\bar{p}_x] \} \]
is Borel, and, using the Jankov-von Neumann theorem, we can choose $\bar{f}_0 \in U_\epsilon(\bar{f})$ such that $\bar{f}_0(x) \in [\bar{p}_x]$ for every $x \in X$.

Fix $a \in M$ and an $n$-tuple of partial automorphisms $\bar{p}$ such that the set
\[ W=\{ x \in X\colon \ a_x=a, \bar{p}_x=\bar{p} \} \]
is not null. Fix $\epsilon_0>0$ such that, for $g \in G$, $g(a) \neq a$ implies that $d(g,e) >\epsilon_0$, and $\bar{f}_1 \in U_{\epsilon_0}(\bar{f}_0)$ implies that $\bar{f}_1(x) \in [\bar{p}]$ on some $W_0 \subseteq W$ (which depends on $\bar{f}_1$) with $\mu(W_0)=\epsilon_0$. Set $V_0=U_{\epsilon_0}(\bar{f}_0)$.

For $k \in \NN$, let 
\[ P_k=\{ \bar{h} \in V_0\colon\, \exists l \geq k \ \mu( \{ x \in X\colon w_l(\bar{h}(x))(a) \neq a\}) > \frac{\epsilon_0}{4} \}. \]
It is easy to see that each $P_k$ is open. We will show that each $P_k$ is also dense in $V_0$, which will imply (as in Lemma \ref{le:NonConv}) that $V_0$ is as required.

Fix $\bar{f}_1 \in V$, $\delta>0$, and a set $W_0$ as above. Because $\bar{f}_1(x) \in [\bar{p}]$ for $x \in W_0$, for every $x \in W_0$ there is an $n$-tuple $\bar{q}_x$ of partial automorphisms of $M$ extending $\bar{p}$, and $u_x \in F_n$ such that $[\bar{q}_x] \subseteq U_\delta(\bar{f}_1(x))$, and $u_x(\bar{q}_x)(a)$ is not defined.

We will find $W_1 \subseteq W_0$ with $\mu(W_1) > \epsilon_0/4$, and a single $u \in F_n$ such that $u(\bar{q}_x)(a)$ is not defined for $x \in W_1$.
For $v \in F_n$, $m \in \NN$, let
\[ A_{v,m}= \{ x \in W_0 \colon \ v^m(\bar{q}_x)(a) \mbox{ is not defined} \}, \]
\[ B_{v,m}= \{ x \in W_0\colon \ v^m(\bar{q}_x)(a)=a \}. \] 
Let $\{v_i\}$ be an enumeration of $F_n$. Note that for every $i \in \NN$ we can fix $m_i \in \NN$ such that
\[ \mu(A_{v_i,m_i} \cup B_{v_i,m_i}) \geq (1-2^{-(i+3)}) \epsilon_0.\]
Put $A_v=\bigcup_i A_{v_i,m_i}$, $B_v=\bigcup_i B_{v_i,m_i}$, and $C=\bigcap_v (A_v \cup B_v)$. Then $\mu(C) > 3/4 \epsilon_0$.

Clearly, for every $i \in \NN$, and $u_i=v^{m_i}_{i}\ldots v^{m_0}_{0}$, either $u_i(\bar{q}_x)(a)$ is not defined  or $u_i(\bar{q}_x)(a)=a$ for all $x \in C$. Now, by the Borel-Cantelli lemma applied to the sequence $\{ B_v \setminus (A_{v_i,m_i} \cup B_{v_i,m_i}) \}_i$, for almost all $x \in B_v \setminus A_v$ there is $i_0$, which depends on $x$, such that $x \in \bigcap_{i \geq i_0} B_{v_i,m_i}$. But for every $x \in W_0$ there are in fact infinitely many $v \in F_n$ such that $v(\bar{q}_x)(a)$ is not defined, so each $\bigcap_{i \geq i_0} B_{v_i,m_i}$ must be empty. It follows that $B_v\setminus A_v$ is null, and, since $\mu(A_v \cup B_v)=\epsilon_0$, that $\mu(\bigcup_i A_{v_i,m_i})=\epsilon_0$. Fix $j_0 \in \NN$ such that $\mu(\bigcup_{i=0}^{j_0} A_{v_i,m_i}) > 3/4 \epsilon_0$. Then, for $W_1=(\bigcup_{i=0}^{j_0} A_{v_i,m_i}) \cap C$ and $u=u_{j_0}$, we have that $\mu(W_1) > \epsilon_0/4$, and $u(\bar{q}_x)(a)$ is not defined for $x \in W_1$.

Fix $l \geq k$ with $w_l=vu$. Arguing as in the proof of Lemma \ref{le:NonConv}, for every $x \in W_1$, we find $\bar{r}_x$ extending $\bar{q}_x$, and such that $w_l(\bar{r}_x)(a)$ is defined and not equal to $a$. It is straightforward to verify that the function $x \mapsto \bar{r}_x$ can be taken to be Borel (just write down the formula defining $\bar{r}_x$ to see that it is Borel). By the Jankov-von Neumann theorem, there exists $\bar{f}_2 \in U_{\delta}(\bar{f}_1)$ such that $\bar{f}_1(x) \in [\bar{r}_x]$ for $x \in W_1$. Because $\mu(W_1) > \epsilon_0/4$, and $d(w_l(\bar{f}_2(x)),\Id) >\epsilon_0$ for $x \in W_1$, we get that $\rho(w_l(\bar{f}_2),\Id) > \epsilon_0/4$, i.e., $\bar{f}_2 \in P_k$. As $\bar{f}_1$ and $\delta$ were arbitrary, this shows that $P_k$ is dense.

The case that there are non-meagerly many $n$-tuples in $G$ generating a non-precompact subgroup is essentially the same. %The last statement of the lemma follows from the fact that $L_0(G)^n$ can be naturally identified with $L_0(G^n)$.
%This implies that the open set of $\bar{g} \in V$ such that $d(w_l(\bar{g}_0),Id)>\epsilon_0/2$ is dense $V$, and, thus, the set of $\bar{g} \in V$ such that $w_k(\bar{g}) \not \rightarrow e$ is comeager in $V$.
\end{proof}

%Recall that we refer to elements of an $n$-tuple $\bar{p}$ by $p^j$, i.e., $\bar{p}=(p^1,\ldots,p^n)$.

\begin{lemma}
\label{le:NonDiscL0}
Let $G$ be a non-archimedean Polish group. Suppose that, for some $n$, $G$ has a dense $n$-diagonal conjugacy class, and the set of $n$-tuples in $G$ generating a non-discrete group is comeager. Then the set of $n$-tuples in $L_0(G)$ generating a non-discrete group is comeager as well.
\end{lemma}

\begin{proof}
We can think of $G$, with a fixed metric $d$, as the automorphism group of the Fra\"{i}ss\'{e} limit $M$ of some Fra\"{i}ss\'{e} class $\mathcal{F}$. As the set of $n$-tuples generating a non-discrete group is a $G_\delta$ in $L_0(G)$, we only need to show that it is dense. Step functions $f_{\bar{a}}$ form a dense set in $L_0(G)$, so it suffices to prove that for every $k$, and every sequence $\bar{p}_j=(p^1_j, \ldots, p^k_j)$, $j \leq n$, of $k$-tuples of partial automorphisms of $M$ there exist $k$-tuples $\bar{q}_j=(q^1_j, \ldots, q^k_j)$, $j \leq n$, of partial automorphisms extending $\bar{p}_j$, respectively, and $w \in F_n$ such that $w(q^i_1, \ldots, q^i_n) \neq e$ but $w(q^i_1, \ldots, q^i_n)(a)=a$ for every $i \leq k$ and $a \in \bigcup_j\dom(p_i^j)$. This is because then  it follows  that for every sequence $\bar{p}_j=(p^1_j, \ldots, p^k_j)$, $j \leq n$, of $k$-tuples of partial automorphisms of $M$, there exist step functions $f_{\bar{a}_1}, \ldots, f_{\bar{a}_n} \in L_0(G)$, where each $\bar{a}_j=(a^1_j, \ldots, a^k_j)$ is a tuple of automorphisms of $M$ extending $(q^1_j, \ldots, q^k_j)$, such that for every $\epsilon>0$ there is $w \in F_n$ with $0<d(w(a^i_1, \ldots, a^i_n),e)<\epsilon$ for $i \leq k$, and thus $0<\rho(w(f_{\bar{a}_1}, \ldots, f_{\bar{a}_1}),e)<\epsilon$. 

Fix such $\bar{p}_j$, $j \leq n$. Let $A_i \in \mathcal{F}$, $i \leq k$, be structures generated by $\bigcup_j (\dom(p_j^i) \cup \rng(p_j^i))$. As $G$ has a dense $n$-diagonal conjugacy class, by \cite[Theorem 2.11]{KR}, the family of all $(A,\bar{p})$, where $A$ is a finite subset of $M$, and $\bar{p}$ is an $n$-tuple of partial automorphisms of $A$, has the joint embedding property. Thus, there exists $A \in \mathcal{F}$, an $n$-tuple $\bar{p}=(p^1, \ldots, p^n)$ of partial automorphisms of $A$, and embeddings $e_i\colon (A_i, (p^i_1, \ldots, p^i_n)) \rightarrow (A, \bar{p})$, $i \leq k$, that is, mappings $e_i\colon A_i \rightarrow A$ that are embeddings of $A_i$, and are such that
\[ e_i \circ p_j^i \subseteq p^j \circ e_i. \]
for every $i \leq k$, $j \leq n$. Because the set of $n$-tuples in $G$ generating a non-discrete group is comeager in $G^n$, there exists an extension $\bar{q}=(q^1, \ldots, q^n)$ of $\bar{p}$, and $w \in F_n$ such that $w(\bar{q}) \neq e$ but $w(\bar{q})(a)=a$ for $a \in A$. As $M$ is a ultrahomogeneous, we can extend each $e_i$ to a partial automorphism $f_i$ in such a way that $f_i^{-1}w(\bar{q}) f_i(a)$ is defined for every $a \in A_i$. Put $\bar{q}_j$, $j \leq n$, to be tuples of the form $(f_i^{-1}q^1f_i, \ldots, f_i^{-1}q^kf_i)$. Then each $\bar{q}_j$ extends $(p^1_j, \ldots, p^k_j)$, and $w(q^i_1, \ldots, q^i_n)\neq e$ but $w(q^i_1, \ldots, q^i_n)(a)=a$ for $i \leq k$, $j \leq n$ and $a \in \dom(p^j_i) \subseteq A_i$.
\end{proof}

For the sake of completeness, we also point out the following fact.
\begin{lemma}
\label{le:NonDiscL0_2}
Let $G$ be a Polish group. Suppose that for some $n$, set of $n$-tuples in $L_0(G)$ generating a non-discrete group is comeager. Then the set of $n$-tuples in $G$ generating a non-discrete group is comeager as well.
\end{lemma}
\begin{proof}
Since the $n$-tuples in $G$ generating a non-discrete subgroup form a $G_\delta$ set in $G^n$, it is enough to show that the set of such tuples is dense in $G^n$.

In $G$, take non-empty open subsets $U_1,\ldots, U_n$. Take some $0<\epsilon<\frac{1}{n}$.
Let $V_i=\{f\in L_0(G)\colon \mu(\{x \in X \colon f(x)\notin U_i\})<\epsilon\}$.
As each $V_i$ is open, we can fix $\bar{f}=(f_1,\ldots, f_n)\in V_1\times\ldots\times V_n$ and a sequence of words
$(w_k)$ such that  $w_k(\bar{f})\to 0$. Without loss of generality, by passing to a subsequence, this convergence is pointwise on a set of full measure. Let $A_i=\{x \in X\colon f_i(x)\notin U_i\}$. But then, for every $x\notin\bigcup_i A_i$ (and by the choice of $\epsilon $ there are such $x$), we have that $w_k(\overline{f}(x))\to 0$, i.e., $\overline{f}(x)\in U_1\times\ldots\times U_n $, and $\overline{f}(x)$ generates a non-discrete subgroup.

\end{proof}

As before, Lemmas \ref{le:NonConvL0} and \ref{le:NonDiscL0} lead to the following:

\begin{theorem}
\label{th:MeagSimClassL0}
Let $M$ be a structure such that algebraic closures of finite subsets of $M$ are finite, and let $G=\Aut(M)$. Suppose that, for some $n$,
\begin{enumerate}
\item $G$ has a dense $n$-diagonal conjugacy class;
\item there are comeagerly many $n$-tuples in $G$ generating a non-precompact and non-discrete subgroup. 
\end{enumerate}
Then each $n$-dimensional class of topological similarity in $L_0(G)$ is meager.
\end{theorem}

\begin{proof}
Suppose that there exists an $n$-tuple in $L_0(G)$ with non-meager class of topological similarity. By Lemma \ref{le:NonConvL0}, it generates a discrete group, i.e., the set of $n$-tuples generating a discrete group is non-meager in $L_0(G)$. Then Lemma \ref{le:NonDiscL0} implies that 
the set of $n$-tuples in $G$ generating a non-discrete group is not comeager, which contradicts the assumptions.
%this set is meager in $L_0(G)$; a contradiction.
\end{proof}

\begin{theorem}
\label{th:TrichotomyL0}
Let $M$ be a structure such that algebraic closures of finite subsets of $M$ are finite, and let $G=\Aut(M)$. Suppose that for some $n$,
$G$ has a dense $n$-diagonal conjugacy class. Then, for every  $n$-tuple $\bar{f}$ in $L_0(G)$:
\begin{itemize}
\item $\langle \bar{f} \rangle$ is precompact, or
\item $\langle \bar{f} \rangle$ is discrete, or
\item the similarity class of $ \bar{f} $ is meager.
\end{itemize}
\end{theorem}

\begin{proof}
%\gre{Ale on chcial, zeby uzyc topologicznej przechodniosci. Poniewaz nie robimy tego, dowod pozostawilem bez zmian.}
Fix an $n$-tuple $\bar{f}$ in $L_0(G)$, and suppose that its similarity class $C$ is non-meager in $L_0(G)^n$. Our assumption on the existence of a dense $n$-diagonal conjugacy class in $G$ implies that there also exists a dense $n$-diagonal conjugacy class in $L_0(G)$ (see Lemma 4 in \cite{KM} and Lemma \ref{le:Orbit}). Then, since the $n$-tuples in $L_0(G)$ with dense conjugacy classes form a $G_\delta$ set, this implies that comeagerly many $n$-tuples have dense conjugacy class. Thus, $C$ must be in fact comeager.

Now, by the last statement of Lemma \ref{le:NonConvL0}, either $\bar{f}$ generates a discrete group or the assumptions of the lemma are not satisfied, i.e., there are non-meagerly many tuples in $G$ that generate a precompact group. But then the set of such tuples must intersect $C$, which implies that $\bar{f}$ also generates a precompact group.
\end{proof}

\begin{corollary}
Every $2$-dimensional class of topological similarity in $L_0(\Aut(M))$, where $M=\QQ$ or $M=\QQ \mathbb{U}_{\prec}$, is meager.
\end{corollary}  

\begin{proof}
As in Corollary \ref{co:Slutsky}, for $M=\QQ$ or $M=\QQ \mathbb{U}_{\prec}$, we observe that $M$ has no algebraicity and every non-identity element in $\Aut(M)$ generates a discrete  group. This implies that every element, and hence also every tuple of elements in $\Aut(M)$, generates a non-precompact group. It follows from \cite[Lemma 3.13]{Sl}, \cite[Lemma 5.7]{Sl} that the set of pairs generating a non-discrete group is comeager in $\Aut(M)^2$. Thus
we only need to show that $\Aut(M)$ has a dense $2$-diagonal conjugacy class.

To see this for $\Aut(\QU_{\prec})$, by \cite[Theorem 2.12]{KR}, it suffices to observe that the family of all $(A,\bar{p})$, where $A$ is a finite subset of $\QU_{\prec}$, and $\bar{p}$ is a pair of partial automorphisms of $A$, has the joint embedding property. But for all fixed pairs $(p_0^1,p_0^2)$, $(p_1^1,p_1^2)$ of partial automorphisms of $\QU_{\prec}$, we can think of each $p_i^j$ as a partial automorphism of some $QU_i \subseteq \QU_{\prec}$, $i=0,1$, such that $x<y$ and $d(x,y)=d_0>0$ for all $x \in QU_0$, $y \in QU_1$. Thus, $q^1=p_0^1 \cup p_1^1$, $q^2=p_0^2 \cup p_1^2$ are also partial automorphisms of $\QU_{\prec}$, and the identity embeds $(QU_0,(p_0^1,p_0^2))$, $(QU_1,(p_1^1,p_1^2))$ into $(QU_0 \cup QU_1, (q^1,q^2))$.

An analogous observation can be used in the case of $\Aut(\QQ )$.
\end{proof}

Finally, we observe that if, for some non-archimedean Polish group $G$, comeagerly many elements in $G$ generate a precompact group, then the situation is the same for $L_0(G)$. This stands in a sharp contrast with the general case: for example, the unit circle $\SSS^1$ is compact, while
$L_0(\SSS^1)$ is not compact and
 it is not hard to see that comeagerly many elements in $L_0(\SSS^1)$ generate a dense subgroup. 
 Indeed, for each $n$, the set of topological 1-generators of $(\SSS^1)^n$ is 
dense in $(\SSS^1)^n$, in fact
every $n$-tuple that consists of independent irrationals topologically generates $(\SSS^1)^n$. The sequence of groups $\SSS^1$,
 $(\SSS^1)^2$, $(\SSS^1)^3$... can be naturally identified with an increasing sequence of subgroups of $L_0(\SSS^1)$ whose union is dense in $L_0(\SSS^1)$. Since for every open set $U$ in $L_0(\SSS^1)$, $\{x\in L_0(\SSS^1)\colon \{x^n\colon n\in\ZZ\}\cap U\neq\emptyset\}$
is open and dense, the set of topological 1-generators in $L_0(\SSS^1)$ is comeager. See also \cite{Gla} and \cite[Theorem 4.2.6]{P}.  

\begin{proposition}
\label{pr:Prec}
Let $G$ be a non-archimedean Polish group. If comeagerly many elements in $G$ generate a precompact group, then comeagerly many elements in $L_0(G)$ generate a precompact group.
\end{proposition}

\begin{proof}
Suppose that the set $P \subseteq G^n$ of all elements generating a precompact group is comeager. Then, by \cite[Lemma 5]{KM}, the set $R \subseteq L_0(G)$ of $f$ such that $f(x) \in P$ for almost all $x \in X$ is comeager in $L_0(G)$. Fix $f \in R$, and a sequence $\{ f^0_k \} \subseteq \gen{f}$. For every $m>0$, we will find $A_m \subseteq X$ with $\mu(A_m)>1-2^{-m}$, and a subsequence $\{ f^m_k \}$ of $\{ f^{m-1}_k \}$ that satisfies $(f^m_k(x))(m')=(f^m_{k'}(x))(m')$ for $x \in A_m$, $k,k' \in \NN$, and $m' \leq m$. It is straightforward to verify that then the sequence $\{ f^m_m \}$ is convergent. 

Fix $m>0$. For every $x \in X$, fix $M_x \in \NN$ such that $(f^{M_x}(x))(m')=m'$ for $m' \leq m$ (note that $M_x$ is well-defined because the evaluation action of each group $\gen{f(x)}$ on $\NN$ has all orbits finite). Thus, there exists $A_m \subseteq X$, and $M \in \NN$ such that $\mu(A_m)>1-2^{-m}$, and $(f^M(x))(m')=m'$ for $x \in A_m$ and $m' \leq m$. 
Using the pigeon hole principle, we can easily choose a subsequence $\{ f^m_k \}$ from $\{ f^{m-1}_k \}$ such that $(f^m_k(x))(m')=(f^m_{k'}(x))(m')$ for every $x \in A_m$, $k,k' \in \NN$ and $m' \leq m$.% in other words,  $\rho(\bar{f}_{k_i},\bar{f}_{k_{i'}}) \leq 2^{-m+1}$ for $i,i' \in \NN$. Continuing in this fashion, we can find a Cauchy subsequence in $\{ \bar{f}_k \}$.
\end{proof}

We do not know whether Proposition \ref{pr:Prec} holds also for tuples of elements from $G$. On the other hand, it is easy to see that, if comeagerly many $n$-tuples in $G$ generate a discrete group, then comeagerly many $n$-tuples in $L_0(G)$ also generate a discrete group. This is because every $\bar{f} \in L_0(G)$ such that $\gen{\bar{f}(x)}$ is discrete for almost all $x \in X$, generates a discrete group.

\section*{Acknowledgement}
We thank Christian Rosendal for pointing out that  Lemma~\ref{le:NonConv} leads to a trichotomy formulated in
Theorem \ref{th:Trichotomy}. We also thank the referee for carefully reading the paper.

\end{document}